\documentclass[a4paper,12pt,reqno]{amsart}
\usepackage{amsfonts}
\usepackage{amsmath}
\usepackage{amssymb}
\usepackage[a4paper]{geometry}
\usepackage{mathrsfs}
\usepackage[colorlinks]{hyperref}
\renewcommand\eqref[1]{(\ref{#1})}
\makeatletter
\newcommand*{\mint}[1]{%
  \mint@l{#1}{}%
}
\newcommand*{\mint@l}[2]{%
  \@ifnextchar\limits{%
    \mint@l{#1}%
  }{%
    \@ifnextchar\nolimits{%
      \mint@l{#1}%
    }{%
      \@ifnextchar\displaylimits{%
        \mint@l{#1}%
      }{%
        \mint@s{#2}{#1}%
      }%
    }%
  }%
}
\newcommand*{\mint@s}[2]{%
  \@ifnextchar_{%
    \mint@sub{#1}{#2}%
  }{%
    \@ifnextchar^{%
      \mint@sup{#1}{#2}%
    }{%
      \mint@{#1}{#2}{}{}%
    }%
  }%
}
\def\mint@sub#1#2_#3{%
  \@ifnextchar^{%
    \mint@sub@sup{#1}{#2}{#3}%
  }{%
    \mint@{#1}{#2}{#3}{}%
  }%
}
\def\mint@sup#1#2^#3{%
  \@ifnextchar_{%
    \mint@sup@sub{#1}{#2}{#3}%
  }{%
    \mint@{#1}{#2}{}{#3}%
  }%
}
\def\mint@sub@sup#1#2#3^#4{%
  \mint@{#1}{#2}{#3}{#4}%
}
\def\mint@sup@sub#1#2#3_#4{%
  \mint@{#1}{#2}{#4}{#3}%
}
\newcommand*{\mint@}[4]{%
  \mathop{}%
  \mkern-\thinmuskip
  \mathchoice{%
    \mint@@{#1}{#2}{#3}{#4}%
        \displaystyle\textstyle\scriptstyle
  }{%
    \mint@@{#1}{#2}{#3}{#4}%
        \textstyle\scriptstyle\scriptstyle
  }{%
    \mint@@{#1}{#2}{#3}{#4}%
        \scriptstyle\scriptscriptstyle\scriptscriptstyle
  }{%
    \mint@@{#1}{#2}{#3}{#4}%
        \scriptscriptstyle\scriptscriptstyle\scriptscriptstyle
  }%
  \mkern-\thinmuskip
  \int#1%
  \ifx\\#3\\\else_{#3}\fi
  \ifx\\#4\\\else^{#4}\fi
}
\newcommand*{\mint@@}[7]{%
  \begingroup
    \sbox0{$#5\int\m@th$}%
    \sbox2{$#5\int_{}\m@th$}%
    \dimen2=\wd0 %
    \let\mint@limits=#1\relax
    \ifx\mint@limits\relax
      \sbox4{$#5\int_{\kern1sp}^{\kern1sp}\m@th$}%
      \ifdim\wd4>\wd2 %
        \let\mint@limits=\nolimits
      \else
        \let\mint@limits=\limits
      \fi
    \fi
    \ifx\mint@limits\displaylimits
      \ifx#5\displaystyle
        \let\mint@limits=\limits
      \fi
    \fi
    \ifx\mint@limits\limits
      \sbox0{$#7#3\m@th$}%
      \sbox2{$#7#4\m@th$}%
      \ifdim\wd0>\dimen2 %
        \dimen2=\wd0 %
      \fi
      \ifdim\wd2>\dimen2 %
        \dimen2=\wd2 %
      \fi
    \fi
    \rlap{%
      $#5%
        \vcenter{%
          \hbox to\dimen2{%
            \hss
            $#6{#2}\m@th$%
            \hss
          }%
        }%
      $%
    }%
  \endgroup
}
\numberwithin{equation}{section}
\theoremstyle{plain}
\newtheorem{thm}{Theorem}[section]

\newtheorem{cor}[thm]{Corollary}

\theoremstyle{definition}

\newtheorem{rem}[thm]{Remark}

\def\S{\mathfrak{S}}
\def\H{\mathbb{H}^{n}}
\def\X{\mathbb{X}}

\begin{document}

 \title[Reverse integral Hardy inequality on metric measure spaces]
{Reverse integral Hardy inequality on metric measure spaces}

\author{Aidyn Kassymov}
\address{
  Aidyn Kassymov:
  \endgraf
   \endgraf
  Department of Mathematics: Analysis, Logic and Discrete Mathematics
  \endgraf
  Ghent University, Belgium
  \endgraf
  and
  \endgraf
  Institute of Mathematics and Mathematical Modeling
  \endgraf
  125 Pushkin str.
  \endgraf
  050010 Almaty
  \endgraf
  Kazakhstan
  \endgraf
  and
  \endgraf
  Al-Farabi Kazakh National University
  \endgraf
   71 Al-Farabi avenue
   \endgraf
   050040 Almaty
   \endgraf
   Kazakhstan
  \endgraf
	{\it E-mail address} {\rm aidyn.kassymov@ugent.be} and {\rm kassymov@math.kz}}

  \author[M. Ruzhansky]{Michael Ruzhansky}
\address{
	Michael Ruzhansky:
	 \endgraf
  Department of Mathematics: Analysis, Logic and Discrete Mathematics
  \endgraf
  Ghent University, Belgium
  \endgraf
  and
  \endgraf
  School of Mathematical Sciences
    \endgraf
    Queen Mary University of London
  \endgraf
  United Kingdom
	\endgraf
  {\it E-mail address} {\rm michael.ruzhansky@ugent.be}}

\author[D. Suragan]{Durvudkhan Suragan}
\address{
	Durvudkhan Suragan:
	\endgraf
	Department of Mathematics
	\endgraf
	School of Science and Technology, Nazarbayev University
    \endgraf
	53 Kabanbay Batyr Ave, Nur-Sultan 010000
	\endgraf
	Kazakhstan
	\endgraf
	{\it E-mail address} {\rm durvudkhan.suragan@nu.edu.kz}}

\thanks{
The authors were supported in parts by the FWO Odysseus Project, the Leverhulme Grant RPG-2017-151 and by EPSRC Grant
EP/R003025/1, as well as NU CRG
091019CRP2120 and the NU FDCRG 240919FD3901.
\
}

     \keywords{Reverse integral Hardy inequality, Reverse Minkowski inequality, metric measure space,
     	 homogeneous Lie group, hyperbolic space, Cartan-Hadamard manifolds.}
 \subjclass{22E30, 43A80.}

     \begin{abstract}
In this note, we obtain a reverse version of the integral Hardy inequality on metric measure spaces. Moreover, we give necessary and sufficient conditions for the weighted reverse Hardy inequality to be true. The main tool in our proof is a continuous version of the reverse Minkowski inequality. Also, we present some consequences of the obtained reverse Hardy inequality on the homogeneous groups, hyperbolic spaces and Cartan-Hadamard manifolds.
     \end{abstract}
     \maketitle

\section{Introduction}
In one of the pioneering papers of Hardy \cite{Har20}, he proved the following (direct) inequality:
\begin{equation}
\int_{a}^{\infty}\frac{1}{x^{p}}\left(\int_{a}^{\infty}f(t)dt\right)^{p}dx\leq\left(\frac{p}{p-1}\right)^{p}\int_{a}^{\infty}f^{p}(x)dx,
\end{equation}
where $f\geq0$, $p>1$, and $a>0$. Today's literature on the development of the extensions of this integral Hardy inequality is very large, see e.g. \cite{Dav99,EE04, KP03, KPS17} and \cite{OK90}. Note that the multi-dimensional version of the integral Hardy inequality was proved in  \cite{DHK}.

In \cite{BH}, the authors obtained the so-called reverse integral Hardy inequality in the following form:
\begin{equation}
\left(\int^{b}_{a}\left(\int_{a}^{x}f(t)dt\right)^{q}u(x)dx\right)^{\frac{1}{q}}\geq C\left(\int^{b}_{a}f^{p}(x)v(x) dx\right)^{\frac{1}{p}}, \end{equation}
and the conjugate reverse integral Hardy inequality
\begin{equation}
\left(\int^{b}_{a}\left(\int_{x}^{b}f(t)dt\right)^{q}u(x)dx\right)^{\frac{1}{q}}\geq C\left(\int^{b}_{a}f^{p}(x)v(x) dx\right)^{\frac{1}{p}}, \end{equation}
where $f\geq0,$ for  some positive weights $u,v$ and $p,q<0$. 
 The reverse Hardy inequalities were also studied in \cite{GKK,KKK08,KK} and \cite{Pro}.

 The main aim of the present paper is to extend the reverse Hardy inequalities to general metric measure spaces. More specifically, we consider metric spaces $\mathbb X$ with a Borel measure $dx$ allowing for the following {\em polar decomposition} at $a\in{\mathbb X}$: we assume that there is a locally integrable function $\lambda \in L^1_{loc}$  such that for all $f\in L^1(\mathbb X)$ we have
   \begin{equation}\label{EQ:polarintro}
   \int_{\mathbb X}f(x)dx= \int_0^{\infty}\int_{\Sigma_{r}} f(r,\omega) \lambda(r,\omega) d\omega dr,
   \end{equation}
    for the set $\Sigma_{r}=\{x\in\mathbb{X}:d(x,a)=r\}\subset \mathbb X$ with a measure on it denoted by $d\omega$, and $(r,\omega)\rightarrow a $ as $r\rightarrow0$.

The condition \eqref{EQ:polarintro} is rather general since we allow the function $\lambda$ to depend on the whole variable $x=(r,\omega)$. The reason to assume \eqref{EQ:polarintro} is that since $\mathbb X$ does not have to have a differentiable structure, the function $\lambda(r,\omega)$ can not be in general obtained as the Jacobian of the polar change of coordinates. However, if such a differentiable structure exists on $\mathbb X$, the condition \eqref{EQ:polarintro} can be obtained as the standard polar decomposition formula.
In particular, let us give several examples of $\mathbb X$ for which the condition \eqref{EQ:polarintro} is satisfied with different expressions for $ \lambda (r,\omega)$:

\begin{itemize}
\item[(I)] Euclidean space $\mathbb{R}^{n}$: $ \lambda (r,\omega)= {r}^{n-1}.$
\item[(II)] Homogeneous groups: $ \lambda (r,\omega)= {r}^{Q-1}$, where $Q$ is the homogeneous dimension of the group. Such groups have been consistently developed by Folland and Stein \cite{FS1}, see also an up-to-date exposition in \cite{FR}.
\item[(III)] Hyperbolic spaces $\mathbb H^n$:  $\lambda(r,\omega)=(\sinh {r})^{n-1}$.
\item[(IV)] Cartan-Hadamard manifolds: Let $K_M$ be the sectional curvature on $(M, g).$ A Riemannian manifold $(M, g)$ is called {\em a Cartan-Hadamard manifold} if it is complete, simply connected and has non-positive sectional curvature, i.e., the sectional curvature $K_M\le 0$ along each plane section at each point of $M$. Let us fix a point $a\in M$ and denote by
$\rho(x)=d(x,a)$ the geodesic distance from $x$ to $a$ on $M$. The exponential map ${\rm exp}_a :T_a M \to  M$ is a diffeomorphism, see e.g. Helgason \cite{DV3}.  Let $J(\rho,\omega)$ be the density function on $M$, see e.g. \cite{DV1}. Then we have the following polar decomposition:
$$
\int_M f(x) dx=\int_0^{\infty}\int_{\mathbb S^{n-1}}f({\rm exp}_{a}(\rho \omega))J(\rho,\omega) \rho^{n-1}d\rho d\omega,
$$
so that we have \eqref{EQ:polarintro} with $\lambda(\rho,\omega)= J(\rho,\omega) \rho^{n-1}.$
\end{itemize}
 In \cite{RV}, the (direct) integral Hardy inequality on metric measure space was established with applications to homogeneous Lie groups, hyperbolic spaces, Cartan-Hadamard manifolds with negative curvature and on general Lie groups with Riemannian distance. Also, on Riemannian manifolds Hardy inequality was obtained in \cite{RYriem}.  In the present paper, we continue the analysis in general setting of metric measure spaces as in \cite{RV} and show the reverse integral Hardy inequality with $q<0$ and $p\in(0,1)$ and we also discuss its consequences for homogeneous Lie groups, hyperbolic spaces and Cartan-Hadamard manifolds with negative curvature.
\section{Main result}
Let us recall briefly the reverse H\"{o}lder's inequality.
\begin{thm}[\cite{AF}, Theorem 2.12, p. 27]\label{Hol}
Let $p\in(0,1)$, so that $p'=\frac{p}{p-1}<0$. If non-negative functions satisfy $f\in L^{p}(\X)$ and $0<\int_{\X}g^{p'}(x)dx<+\infty,$ we have
\begin{equation}\label{Holin}
\int_{\X}f(x)g(x)dx\geq\left(\int_{\X}f^{p}(x)dx\right)^{\frac{1}{p}}\left(\int_{\X}g^{p'}(x)dx\right)^{\frac{1}{p'}}.
\end{equation}
\end{thm}
Let us present the reverse integral Minkowski inequality (or a continuous version of reverse Minkowski inequality).
\begin{thm}
Let $\mathbb X, \mathbb Y$ be metric measure spaces and let $F=F(x,y)\in \mathbb X\times \mathbb Y$ be a non-negative measurable function. Then we have
\begin{equation}\label{minkow}
\left[\int_{\mathbb X}\left(\int_{\mathbb Y}F(x,y)dy\right)^{q}dx\right]^{\frac{1}{q}}\geq\int_{\mathbb Y}\left(\int_{\mathbb X}F^{q}(x,y)dx\right)^{\frac{1}{q}}dy,\; q<0.
\end{equation}
\end{thm}
\begin{proof}
Let us consider the following function:
\begin{equation}
A(x):=\int_{\mathbb Y}F(x,y)dy,
\end{equation}
so we have
\begin{equation}
A^{q}(x)=\left(\int_{\mathbb Y}F(x,y)dy\right)^{q}.
\end{equation}
By integrating over $\mathbb X$ both sides and by using reverse H\"{o}lder's inequality (Theorem \ref{Hol}), we obtain
\begin{equation}
\begin{split}
\int_{\mathbb X}A^{q}(x)dx&=\int_{\mathbb X}A^{q-1}(x)A(x)dx\\&
=\int_{\mathbb X}A^{q-1}(x)\int_{\mathbb Y}F(x,y)dydx\\&
=\int_{\mathbb Y}\int_{\mathbb X}A^{q-1}(x)F(x,y)dxdy\\&
\stackrel{\eqref{Holin}}\geq \int_{\mathbb Y}\left(\int_{\mathbb X}A^{q-1\frac{q}{q-1}}(x)dx\right)^{\frac{q-1}{q}}\left(\int_{\mathbb X}F^{q}(x,y)dx\right)^{\frac{1}{q}}dy\\&
= \left(\int_{\mathbb X}A^{q}(x)dx\right)^{\frac{q-1}{q}}\int_{\mathbb Y}\left(\int_{\mathbb X}F^{q}(x,y)dx\right)^{\frac{1}{q}}dy.
\end{split}
\end{equation}
From this,  we get
\begin{equation}
\left[\int_{\mathbb X}\left(\int_{\mathbb Y}F(x,y)dy\right)^{q}dx\right]^{\frac{1}{q}}\geq\int_{\mathbb Y}\left(\int_{\mathbb X}F^{q}(x,y)dx\right)^{\frac{1}{q}}dy,
\end{equation}
proving \eqref{minkow}.
\end{proof}
\begin{rem}
In our sense, the negative exponent $q<0$ of  $0$, we understand in the following form:
\begin{equation}
0^{q}=(+\infty)^{-q}=+\infty,\,\,\,\,\text{and}\,\,\,\,\,\,0^{-q}=(+\infty)^{q}=0.
\end{equation}
\end{rem}
We denote by $B(a, r)$ the ball in $\X$ with centre $a$ and radius $r$, i.e
$$B(a, r) := \{x \in \X : d(x,a) < r\},$$
where $d$ is the metric on $\X$. Once and for all we will fix some point $a \in \X$, and we
will write
\begin{equation}
|x|_{a}:= d(a, x).
\end{equation}
Now we prove the reverse integral Hardy inequality on a metric measure space.
\begin{thm}
Assume that $p\in(0,1)$ and $q<0$. Let $\mathbb X$ be a metric measure space with a polar decomposition at $a\in \X$. Suppose that $u,v>0$ are locally integrable functions on $\mathbb X$. Then the inequality
\begin{equation}\label{har1}
\left[\int_{\mathbb X}\left(\int_{B(a,|x|_{a})}f(y)dy\right)^{q}u(x)dx\right]^{\frac{1}{q}}\geq C(p,q)\left(\int_{\mathbb X}f^{p}(x)v(x)dx\right)^{\frac{1}{p}}
\end{equation}
holds for some $C(p,q)>0$ and for all non-negative real-valued measurable functions $f$, if and only if
\begin{equation}\label{D1}
0< D_{1}:=\inf_{x\neq a}\left[\left(\int_{\mathbb X\setminus B(a,|x|_{a})}u(y)dy\right)^{\frac{1}{q}}\left(\int_{B(a,|x|_{a})}v^{1-p'}(y)dy\right)^{\frac{1}{p'}}\right].
\end{equation}
Moreover, the biggest constant $C(p,q)$ in \eqref{har1} has the following relation to $D_{1}$:
\begin{equation}\label{C}
D_{1}\geq C(p,q)\geq \left(\frac{p'}{p'+q}\right)^{-\frac{1}{q}}\left(\frac{q}{p'+q}\right)^{-\frac{1}{p'}}D_{1}.
\end{equation}
\end{thm}
\begin{proof}
Let us divide proof of this theorem in several steps.

\textbf{Step 1.} Let us denote   $g(x):=f(x)v^{\frac{1}{p}}(x)$. Let  $\frac{1}{p}+\frac{1}{p'}=1,$ $\alpha\in\left(0,-\frac{1}{p'}\right)$ and $z(x)=v^{-\frac{1}{p}}(x)$. Let us denote, using \eqref{EQ:polarintro},
\begin{align}
&
V(x):=\int_{B(a,|x|_{a})}v^{-\frac{p'}{p}}(y)dy=\int_{B(a,|x|_{a})}z^{p'}(y)dy,\\&
H_{1}(s):=\int_{\sum_{s}}\lambda(s,\sigma)g(s,\sigma)z(s,\sigma)d\sigma,\label{h1}\\&
H_{2}(s):=\int_{\sum_{s}}\lambda(s,\sigma)z^{p'}(s,\sigma)V^{\alpha p'}(s,\sigma)d\sigma,\label{h2}\\&
H_{3}(s):=\int_{\sum_{s}}\lambda(s,\sigma)g^{p}(s,\sigma)V^{-\alpha p}(s,\sigma)d\sigma,\label{h3}\\&
U(r):=\int_{\sum_{r}}\lambda(r,\omega)u(r,\omega)d\omega.\label{u}
\end{align}

After some calculation, we compute, using reverse H\"{o}lder's inequality (Theorem \ref{Hol}),
\begin{equation}
\begin{split}
A:&=\int_{\mathbb X}\left(\int_{B(a,|x|_{a})}f(y)dy\right)^{q}u(x)dx=\int_{\mathbb X}\left(\int_{B(a,|x|_{a})}g(y)z(y)dy\right)^{q}u(x)dx\\&
=\int_{\mathbb X}\left(\int_{B(a,|x|_{a})}g(y)z(y)dy\right)^{p}\left(\int_{B(a,|x|_{a})}g(y)z(y)dy\right)^{q-p}u(x)dx\\&
=\int_{\mathbb X}\left(\int_{B(a,|x|_{a})}g(y)V^{-\alpha}(y)V^{\alpha}(y)z(y)dy\right)^{p}\left(\int_{B(a,|x|_{a})}g(y)z(y)dy\right)^{q-p}u(x)dx\\&
\geq \int_{\mathbb X}\left(\int_{B(a,|x|_{a})}g^{p}(y)V^{-\alpha p}(y)dy\right)\left(\int_{B(a,|x|_{a})}z^{p'}(y)V^{\alpha p'}(y)dy\right)^{\frac{p}{p'}}\\&
\times\left(\int_{B(a,|x|_{a})}g(y)z(y)dy\right)^{q-p}u(x)dx\\&
=\int_{0}^{\infty}U(r)\left(\int_{0}^{r}H_{1}(s)ds\right)^{q-p}\left(\int_{0}^{r}H_{2}(s)ds\right)^{\frac{p}{p'}}\left(\int_{0}^{r}H_{3}(s)ds\right)dr.
\end{split}
\end{equation}

Let us denote by $\tilde{H}_{2}(s):=\int_{\sum_{s}}\lambda(s,\sigma)z^{p'}(s,\sigma)d\sigma.$ Then we have
\begin{align}
    \left(\int_{0}^{r}H_{2}(s)ds\right)^{\frac{p}{p'}}&=\left(\int_{0}^{r}\int_{\sum_{s}}\lambda(s,\sigma)z^{p'}(s,\sigma)V^{\alpha p'}(s,\sigma)dsd\sigma\right)^{\frac{p}{p'}}\nonumber \\&
=\left(\int_{0}^{r}\int_{\sum_{s}}\lambda(s,\sigma)z^{p'}(s,\sigma)\left(\int_{0}^{s}\int_{\sum_{\rho}}\lambda(\rho,\sigma_{1})z^{p'}(\rho,\sigma_{1})d\rho d\sigma_{1}\right)^{\alpha p'}dsd\sigma\right)^{\frac{p}{p'}}\nonumber\\&
=\left(\int_{0}^{r}\tilde{H}_{2}(s)\left(\int_{0}^{s}\tilde{H}_{2}(\rho)d\rho \right)^{\alpha p'}ds\right)^{\frac{p}{p'}}\label{2.11}\\&
=\left(\int_{0}^{r}\left(\int_{0}^{s}\tilde{H}_{2}(\rho)d\rho \right)^{\alpha p'}d_{s}\left(\int_{0}^{s}\tilde{H}_{2}(\rho)d\rho\right)\right)^{\frac{p}{p'}}\nonumber\\&
\stackrel{1+\alpha p'>0}=\frac{1}{(1+\alpha p')^{\frac{p}{p'}}}\left(\left(\int_{0}^{s}\tilde{H}_{2}(\rho)d\rho\right)^{1+\alpha p'}\big{|}_{0}^{r}\right)^{\frac{p}{p'}}\nonumber\\&
\stackrel{1+\alpha p'>0}=\frac{1}{(1+\alpha p')^{\frac{p}{p'}}}\left(\int_{0}^{r}\tilde{H}_{2}(\rho)d\rho\right)^{\frac{p(1+\alpha p')}{p'}}\nonumber\\&
=\frac{V_{1}^{\frac{p(1+\alpha p')}{p'}}(r)}{(1+\alpha p')^{\frac{p}{p'}}}\nonumber,
\end{align}
where $V_{1}(r)=\int_{0}^{r}\tilde{H}_{2}(\rho)d\rho$. By using this fact and reverse H\"{o}lder's inequality with $\frac{p}{q}+\frac{q-p}{q}=1$, we obtain
\begin{align*}
A&\geq\int_{0}^{\infty}\left(\int_{0}^{r}H_{3}(s)ds\right)U(r)\left(\int_{0}^{r}H_{1}(s)ds\right)^{q-p}\left(\int_{0}^{r}H_{2}(s)ds\right)^{\frac{p}{p'}}dr\\&
\stackrel{\eqref{2.11}}=\frac{1}{(1+\alpha p')^{\frac{p}{p'}}}\int_{0}^{\infty}\left(\int_{0}^{r}H_{3}(s)ds\right)U(r)\left(\int_{0}^{r}H_{1}(s)ds\right)^{q-p}V_{1}^{\frac{p(1+\alpha p')}{p'}}(r)dr\\&
=\frac{1}{(1+\alpha p')^{\frac{p}{p'}}}\int_{0}^{\infty}U^{\frac{p}{q}}(r)\left(\int_{0}^{r}H_{3}(s)ds\right)V_{1}^{\frac{p(1+\alpha p')}{p'}}(r)\left(\int_{0}^{r}H_{1}(s)ds\right)^{q-p}U^{\frac{q-p}{q}}dr\\&
\geq\frac{1}{(1+\alpha p')^{\frac{p}{p'}}}\left(\int_{0}^{\infty}\left(\int_{0}^{r}H_{3}(s)ds\right)^{\frac{q}{p}}U(r)V_{1}^{\frac{q(1+\alpha p')}{p'}}(r)dr\right)^{\frac{p}{q}}\\&
\times\left(\int_{0}^{\infty}\left(\int_{0}^{r}H_{1}(s)ds\right)^{q}U(r)dr\right)^{\frac{q-p}{q}}\\&
=\frac{1}{(1+\alpha p')^{\frac{p}{p'}}}\left(\int_{0}^{\infty}U(r)\left(\int_{0}^{r}H_{3}(s)ds\right)^{\frac{q}{p}}V_{1}^{\frac{q(1+\alpha p')}{p'}}(r)dr\right)^{\frac{p}{q}}\\&\times\left(\int_{\mathbb X}\left(\int_{B(a,|x|_{a})}g(y)z(y)dy\right)^{q}u(x)dx\right)^{\frac{q-p}{q}}\\&
=\frac{A^{\frac{q-p}{q}}}{(1+\alpha p')^{\frac{p}{p'}}}\left(\int_{0}^{\infty}U(r)\left(\int_{0}^{r}H_{3}(s)ds\right)^{\frac{q}{p}}V_{1}^{\frac{q(1+\alpha p')}{p'}}(r)dr\right)^{\frac{p}{q}}.
\end{align*}
Therefore,
$$A^{\frac{p}{q}}\geq\frac{1}{(1+\alpha p')^{\frac{p}{p'}}}\left(\int_{0}^{\infty}U(r)\left(\int_{0}^{r}H_{3}(s)ds\right)^{\frac{q}{p}}V_{1}^{\frac{q(1+\alpha p')}{p'}}(r)dr\right)^{\frac{p}{q}}.$$

Let us treat the following integral with reverse Minkowski inequality with exponent $\frac{q}{p}<0$ , so that we obtain
\begin{align*}
&\left(\int_{0}^{\infty}U(r)\left(\int_{0}^{r}H_{3}(s)ds\right)^{\frac{q}{p}}V_{1}^{\frac{q(1+\alpha p')}{p'}}(r)dr\right)^{\frac{p}{q}}\\&
=\left(\int_{0}^{\infty}\left(\int_{0}^{r}U^{\frac{p}{q}}(r)H_{3}(s)V^{\frac{(1+\alpha p')p}{p'}}(r)ds\right)^{\frac{q}{p}}dr\right)^{\frac{p}{q}}\\&
=\left(\int_{0}^{\infty}\left(\int_{0}^{\infty}U^{\frac{p}{q}}(r)H_{3}(s)V_{1}^{\frac{(1+\alpha p')p}{p'}}(r)\chi_{\{s<r\}}ds\right)^{\frac{q}{p}}dr\right)^{\frac{p}{q}}\\&
\stackrel{\eqref{minkow}}\geq \int_{0}^{\infty}H_{3}(s)\left(\int_{s}^{\infty}U(r)V_{1}^{\frac{q(1+\alpha p')}{p'}}(r)dr\right)^{\frac{p}{q}}ds\\&
=\int_{\mathbb X}g^{p}(y)V^{-\alpha p}(y)\left(\int_{\mathbb X\setminus B(a,|y|_{a})}u(x)V^{\frac{q(1+\alpha p')}{p'}}(x)dx\right)^{\frac{p}{q}}dy\\&
\geq D^{p}(\alpha)\int_{\mathbb X}g^{p}(y)dy,
\end{align*}
where $D(\alpha):= \inf_{x\neq a}D(x,\alpha)=\inf_{x\neq a}V^{-\alpha }(x)\left(\int_{\mathbb X\setminus B(a,|x|_{a})}u(y)V^{\frac{q(1+\alpha p')}{p'}}(y)dy\right)^{\frac{1}{q}}$ and $\chi$ is the cut-off function.
Then we obtain
\begin{align*}
A^{\frac{p}{q}}&=\left(\int_{\mathbb X}\left(\int_{B(a,|x|_{a})}f(y)dy\right)^{q}u(x)dx\right)^{\frac{p}{q}}\geq \frac{D^{p}(\alpha)}{(1+\alpha p')^{\frac{p}{p'}}} \int_{\mathbb X}g^{p}(y)dy\\&
=\frac{D^{p}(\alpha)}{(1+\alpha p')^{\frac{p}{p'}}} \int_{\mathbb X}f^{p}(y)v(y)dy.
\end{align*}

\textbf{Step 2.} Let us recall  $D_{1}$, given in the following form:
\begin{equation}
0< D_{1}=\inf_{x\neq a}\left[\left(\int_{\mathbb X\setminus B(a,|x|_{a})}u(x)dx\right)^{\frac{1}{q}}\left(\int_{B(a,|x|_{a})}v^{1-p'}(y)dy\right)^{\frac{1}{p'}}\right].
\end{equation}

Let us note a relation between $V$ and $V_{1}$,
\begin{align}
V(x)=\int_{B(a,|x|_{a})}v^{-\frac{p'}{p}}dx&=\int_{B(a,|x|_{a})}z^{p'}dx\nonumber\\&
=\int_{0}^{|x|_{a}}\int_{\sum_{r}}z^{p'}(r,\omega)\lambda(r,\omega)drd\omega\label{VV1}\\&
=\int_{0}^{|x|_{a}}\tilde{H}_{2}(r)dr\nonumber\\&
=:V_{1}(|x|_{a})\nonumber,
\end{align}
where, as before, $\tilde{H}_{2}(r)=\int_{\sum_{r}}z^{p'}(r,\omega)\lambda(r,\omega)d\omega$.
Then let us calculate the following integral:
\begin{equation}
\begin{split}
I&=\int_{\mathbb X\setminus B(a,|x|_{a})}u(y)V^{\frac{q(1+\alpha p')}{p'}}(y)dy
=\int_{|x|_{a}}^{\infty}\int_{\sum_{r}}\lambda(r,\omega)u(r,\omega)V^{\frac{q(1+\alpha p')}{p'}}_{1}(r)drd\omega\\&
=\int_{|x|_{a}}^{\infty}U(r)V^{\frac{q(1+\alpha p')}{p'}}_{1}(r)dr=\int_{|x|_{a}}^{\infty}V^{\frac{q(1+\alpha p')}{p'}}_{1}(r)d_{r}\left(-\int_{r}^{\infty}U(s)ds\right)\\&
=-V^{\frac{q(1+\alpha p')}{p'}}_{1}(r)\int_{r}^{\infty}U(s)ds\big{|}_{|x|_{a}}^{\infty}\\&
+\frac{q(1+\alpha p')}{p'}\int_{|x|_{a}}^{\infty}\left(\int_{r}^{\infty}U(s)ds\right)V_{1}^{\frac{q(1+\alpha p')}{p'}-1}(r)dV_{1}(r)\\&
\stackrel{\frac{q}{p'}>0}=V^{\frac{q(1+\alpha p')}{p'}}_{1}(|x|_{a})\int_{|x|_{a}}^{\infty}U(s)ds\\&+\frac{q(1+\alpha p')}{p'}\int_{|x|_{a}}^{\infty}\left(\int_{r}^{\infty}U(s)ds\right)V_{1}^{\frac{q(1+\alpha p')}{p'}-1}(r)dV_{1}(r)\\&
=V^{\frac{q}{p'}}_{1}(|x|_{a})\left(\int_{|x|_{a}}^{\infty}U(s)ds\right) V^{\alpha q}_{1}(|x|_{a})\\&
+\frac{q(1+\alpha p')}{p'}\int_{|x|_{a}}^{\infty}\left(\int_{r}^{\infty}U(s)ds\right)V_{1}^{\frac{q}{p'}}(r)V_{1}^{\alpha q-1}(r)dV_{1}(r)\\&
\leq D_{1}^{q}V_{1}^{\alpha q}(|x|_{a})+\frac{q(1+\alpha p')D_{1}^{q}}{p'}\int_{|x|_{a}}^{\infty}V_{1}^{\alpha q-1}(r)dV_{1}(r)\\&
=D_{1}^{q}V_{1}^{\alpha q}(|x|_{a})+\frac{(1+\alpha p')D^{q}_{1}}{\alpha p'}V^{\alpha q}_{1}(r)\big{|}_{|x|_{a}}^{\infty}\\&
=D_{1}^{q}V_{1}^{\alpha q}(|x|_{a})+\lim_{r\rightarrow \infty}\frac{(1+\alpha p')D^{q}_{1}}{\alpha p'}V^{\alpha q}_{1}(r)-\frac{(1+\alpha p')D^{q}_{1}}{\alpha p'}V^{\alpha q}_{1}(|x|_{a})\\&
\stackrel{\frac{(1+\alpha p')D^{q}_{1}}{\alpha p'}<0}\leq D_{1}^{q}V_{1}^{\alpha q}(|x|_{a})-\frac{(1+\alpha p')D^{q}_{1}}{\alpha p'}V^{\alpha q}_{1}(|x|_{a})\\&
\stackrel{\eqref{VV1}}= -\frac{1}{\alpha p'}D_{1}^{q}V^{\alpha q}(x).
\end{split}
\end{equation}
Then we have $I=D^{q}(x,\alpha)V^{\alpha q}(x)\leq -\frac{1}{\alpha p'}D_{1}^{q}V^{\alpha q}(x)$. Consequently,
$$D(x,\alpha)\geq(-\alpha p')^{-\frac{1}{q}}D_{1},$$
it means $$D(\alpha)\geq(-\alpha p')^{-\frac{1}{q}}D_{1}.$$
Finally, we obtain
$$A^{\frac{1}{q}}\geq\frac{D_{1}(-\alpha p')^{-\frac{1}{q}}}{(1+\alpha p')^{\frac{1}{p'}}} \left(\int_{\mathbb X}f^{p}(y)v(y)dy\right)^{\frac{1}{p}}.$$
Let us consider the function $k(\alpha):=\frac{(-\alpha p')^{-\frac{1}{q}}}{(1+\alpha p')^{\frac{1}{p'}}}=(-\alpha p')^{-\frac{1}{q}}(1+\alpha p')^{-\frac{1}{p'}}$, where $\alpha\in\left(0,-\frac{1}{p'}\right)$. Firstly, let us find extremum of this function. We have
\begin{equation}
\begin{split}
\frac{d k(\alpha)}{d\alpha}&=-\frac{1}{q}(-p')(-\alpha p')^{-\frac{1}{q}-1}(1+\alpha p')^{-\frac{1}{p'}}+\left(-\frac{1}{p'}\right)p'(1+\alpha p')^{-\frac{1}{p'}-1}(-\alpha p')^{-\frac{1}{q}}\\&
=p'(-\alpha p')^{-\frac{1}{q}-1}(1+\alpha p')^{-\frac{1}{p'}-1}\left(\frac{(1+\alpha p')}{q}+\alpha\right)\\&
=\frac{p'}{q}(-\alpha p')^{-\frac{1}{q}-1}(1+\alpha p')^{-\frac{1}{p'}-1}\left(\alpha(p'+q)+1\right)=0,
\end{split}
\end{equation}
it implies that its solution is
$$\alpha_{1}=-\frac{1}{p'+q}\in\left(0,-\frac{1}{p'}\right).$$
By taking the second derivative of $k(\alpha)$ at the point $\alpha_{1}$ and by denoting $k_{1}(\alpha)=(-\alpha p')^{-\frac{1}{q}-1}(1+\alpha p')^{-\frac{1}{p'}-1}$, we obtain
\begin{equation}
\begin{split}
\frac{d^{2} k(\alpha)}{d\alpha^{2}}\big{|}_{\alpha=\alpha_{1}}&=\left(\frac{p'}{q}(-\alpha p')^{-\frac{1}{q}-1}(1+\alpha p')^{-\frac{1}{p'}-1}\left(\alpha(p'+q)+1\right)\right)'\big{|}_{\alpha=\alpha_{1}}\\&
=\frac{p'}{q}\left(k_{1}(\alpha)\left(\alpha(p'+q)+1\right)\right)'\big{|}_{\alpha=\alpha_{1}}\\&
=\frac{p'}{q}\left(\frac{dk_{1}(\alpha)}{d\alpha}\left(\alpha(p'+q)+1\right)+(p'+q)k_{1}(\alpha)\right)\big{|}_{\alpha=\alpha_{1}}\\&
=\frac{p'}{q}\left(\frac{dk_{1}(\alpha)}{d\alpha}\big{|}_{\alpha=\alpha_{1}}\underbrace{\left(\alpha_{1}(p'+q)+1\right)}_{=0}+(p'+q)k_{1}(\alpha_{1})\right)\\&
=\frac{p'(p'+q)}{q}k_{1}(\alpha_{1})=\frac{p'(p'+q)}{q}(-\alpha_{1} p')^{-\frac{1}{q}-1}(1+\alpha_{1} p')^{-\frac{1}{p'}-1}\\&
=\underbrace{\frac{p'(p'+q)}{q}}_{<0}\underbrace{\left(\frac{p'}{p'+q}\right)^{-\frac{1}{q}-1}}_{>0}\underbrace{\left(\frac{q}{p'+q}\right)^{-\frac{1}{p'}-1}}_{>0}<0.
\end{split}
\end{equation}
It means, function $k(\alpha)$ has supremum at the point $\alpha=\alpha_{1}$. Then, the biggest constant has the following relationship $C(p,q)\geq \left(\frac{p'}{p'+q}\right)^{-\frac{1}{q}}\left(\frac{q}{p'+q}\right)^{-\frac{1}{p'}}D_{1}$.

\textbf{Step 3.} Let us give a necessity condition of  inequality \eqref{har1}. By using \eqref{har1} and $f(x)=v^{-\frac{p'}{p}}(x)\chi_{\{(0,t)\}}(|x|_{a})$, we compute
\begin{equation}
\begin{split}
C(p,q)&\leq \left[\int_{\mathbb X}\left(\int_{B(a,|x|_{a})}f(y)dy\right)^{q}u(x)dx\right]^{\frac{1}{q}} \left[\int_{\mathbb X}f^{p}(y)v(y)dx\right]^{-\frac{1}{p}}\\&
=\left[\int_{\mathbb X}\left(\int_{|y|_{a}\leq t}v^{1-p'}(y)dy\right)^{q}u(x)dx\right]^{\frac{1}{q}}\left[\int_{|y|_{a}\leq t}v^{-p'}(y)v(y)dx\right]^{-\frac{1}{p}}\\&
\stackrel{q<0}\leq\left[\int_{|x|_{a}\geq t}\left(\int_{|y|_{a}\leq t}v^{1-p'}(y)dy\right)^{q}u(x)dx\right]^{\frac{1}{q}}\left[\int_{|y|_{a}\leq t}v^{-p'}(y)v(y)dx\right]^{-\frac{1}{p}}\\&
=\left[\int_{|x|_{a}\geq t}u(x)dx\right]^{\frac{1}{q}}\left[\int_{|y|_{a}\leq t}v^{-p'}(y)v(y)dx\right]^{\frac{1}{p'}},
\end{split}
\end{equation}
 which gives $D_{1}\geq C(p,q)$.
\end{proof}
Let us give conjugate reverse integral Hardy inequality.
\begin{thm}
Assume that $p\in(0,1)$ and $q<0$. Let $\mathbb X$ be a metric measure space with a polar decomposition at $a$. Suppose that $u,v>0$ are locally integrable functions on $\mathbb X$. Then the inequality
\begin{equation}\label{har*}
\left[\int_{\mathbb X}\left(\int_{\mathbb X\setminus B(a,|x|_{a})}f(y)dy\right)^{q}u(x)dx\right]^{\frac{1}{q}}\geq C(p,q)\left(\int_{\mathbb X}f^{p}(x)v(x)dx\right)^{\frac{1}{p}}
\end{equation}
holds for some $C(p,q)>0$ and for all non-negative real-valued measurable functions $f$, if only if
\begin{equation}\label{D2}
0< D_{2}:=\inf_{x\neq a}\left[\left(\int_{ B(a,|x|_{a})}u(y)dy\right)^{\frac{1}{q}}\left(\int_{\mathbb X\setminus B(a,|x|_{a})}v^{1-p'}(y)dy\right)^{\frac{1}{p'}}\right].
\end{equation}
Moreover, the biggest constant $C(p,q)$ in \eqref{har*} has the following relation to $D_{2}$:
\begin{equation}
D_{2}\geq C(p,q)\geq \left(\frac{p'}{p'+q}\right)^{-\frac{1}{q}}\left(\frac{q}{p'+q}\right)^{-\frac{1}{p'}}D_{2}.
\end{equation}

\end{thm}
\begin{proof}
Proof of this theorem is similar to the previous case.
Let us divide proof of this theorem in several steps.

\textbf{Step 1.} Let us denote   $g(x):=f(x)v^{\frac{1}{p}}(x)$. Let  $\frac{1}{p}+\frac{1}{p'}=1,$ $\alpha\in\left(0,-\frac{1}{p'}\right)$ and $z(x)=v^{-\frac{1}{p}}(x)$. Let us denote,
$$G(x):=\int_{\X\setminus B(a,|x|_{a})}v^{-\frac{p'}{p}}(y)dy=\int_{\X\setminus B(a,|x|_{a})}z^{p'}(y)dy.$$
After some calculation, we compute, using reverse H\"{o}lder's inequality (Theorem \ref{Hol}),
\begin{equation}
\begin{split}
B:&=\int_{\mathbb X}\left(\int_{\X\setminus B(a,|x|_{a})}f(y)dy\right)^{q}u(x)dx=\int_{\mathbb X}\left(\int_{\X\setminus B(a,|x|_{a})}g(y)z(y)dy\right)^{q}u(x)dx\\&
=\int_{\mathbb X}\left(\int_{\X\setminus B(a,|x|_{a})}g(y)z(y)dy\right)^{p}\left(\int_{\X\setminus B(a,|x|_{a})}g(y)z(y)dy\right)^{q-p}u(x)dx\\&
=\int_{\mathbb X}\left(\int_{\X\setminus B(a,|x|_{a})}g(y)G^{-\alpha}(y)G^{\alpha}(y)z(y)dy\right)^{p}\left(\int_{\X\setminus B(a,|x|_{a})}g(y)z(y)dy\right)^{q-p}u(x)dx\\&
\geq \int_{\mathbb X}\left(\int_{\X\setminus B(a,|x|_{a})}g^{p}(y)G^{-\alpha p}(y)dy\right)\left(\int_{\X\setminus B(a,|x|_{a})}z^{p'}(y)G^{\alpha p'}(y)dy\right)^{\frac{p}{p'}}\\&
\times\left(\int_{\X\setminus B(a,|x|_{a})}g(y)z(y)dy\right)^{q-p}u(x)dx\\&
=\int_{0}^{\infty}U(r)\left(\int_{r}^{\infty}H_{1}(s)ds\right)^{q-p}\left(\int_{r}^{\infty}H_{2}(s)ds\right)^{\frac{p}{p'}}\left(\int_{r}^{\infty}H_{3}(s)ds\right)dr,
\end{split}
\end{equation}
where $U(r),H_{i}(s), \,i=1,2,3,$ are defined in \eqref{h1}-\eqref{u}. Let us denote by $\tilde{H}_{2}(s):=\int_{\sum_{s}}\lambda(s,\sigma)z^{p'}(s,\sigma)d\sigma.$ Then we have
\begin{align}
    \left(\int_{r}^{\infty}H_{2}(s)ds\right)^{\frac{p}{p'}}&=\left(\int_{r}^{\infty}\int_{\sum_{s}}\lambda(s,\sigma)z^{p'}(s,\sigma)V^{\alpha p'}(s,\sigma)dsd\sigma\right)^{\frac{p}{p'}}\nonumber \\&
=\left(\int_{r}^{\infty}\int_{\sum_{s}}\lambda(s,\sigma)z^{p'}(s,\sigma)\left(\int_{r}^{\infty}\int_{\sum_{\rho}}\lambda(\rho,\sigma_{1})z^{p'}(\rho,\sigma_{1})d\rho d\sigma_{1}\right)^{\alpha p'}dsd\sigma\right)^{\frac{p}{p'}}\nonumber\\&
=\left(\int_{r}^{\infty}\tilde{H}_{2}(s)\left(\int_{s}^{\infty}\tilde{H}_{2}(\rho)d\rho \right)^{\alpha p'}ds\right)^{\frac{p}{p'}}\label{2.11*}\\&
=\left(\int_{r}^{\infty}\left(\int_{s}^{\infty}\tilde{H}_{2}(\rho)d\rho \right)^{\alpha p'}d_{s}\left(-\int_{s}^{\infty}\tilde{H}_{2}(\rho)d\rho\right)\right)^{\frac{p}{p'}}\nonumber\\&
=\left(-\int_{r}^{\infty}\left(\int_{s}^{\infty}\tilde{H}_{2}(\rho)d\rho \right)^{\alpha p'}d_{s}\left(\int_{s}^{\infty}\tilde{H}_{2}(\rho)d\rho\right)\right)^{\frac{p}{p'}}\nonumber\\&
\stackrel{1+\alpha p'>0}=\frac{1}{(1+\alpha p')^{\frac{p}{p'}}}\left(-\left(\int_{s}^{\infty}\tilde{H}_{2}(\rho)d\rho\right)^{1+\alpha p'}\big{|}_{r}^{\infty}\right)^{\frac{p}{p'}}\nonumber\\&
\stackrel{1+\alpha p'>0}=\frac{1}{(1+\alpha p')^{\frac{p}{p'}}}\left(\int_{r}^{\infty}\tilde{H}_{2}(\rho)d\rho\right)^{\frac{p(1+\alpha p')}{p'}}\nonumber\\&
=\frac{G_{1}^{\frac{p(1+\alpha p')}{p'}}(r)}{(1+\alpha p')^{\frac{p}{p'}}}\nonumber,
\end{align}
where $G_{1}(r)=\int_{0}^{r}\tilde{H}_{2}(\rho)d\rho$. By using this fact and reverse H\"{o}lder's inequality with $\frac{p}{q}+\frac{q-p}{q}=1$, we obtain
\begin{align*}
B&\geq\int_{0}^{\infty}\left(\int_{r}^{\infty}H_{3}(s)ds\right)U(r)\left(\int_{r}^{\infty}H_{1}(s)ds\right)^{q-p}\left(\int_{r}^{\infty}H_{2}(s)ds\right)^{\frac{p}{p'}}dr\\&
\stackrel{\eqref{2.11*}}=\frac{1}{(1+\alpha p')^{\frac{p}{p'}}}\int_{0}^{\infty}\left(\int_{r}^{\infty}H_{3}(s)ds\right)U(r)\left(\int_{r}^{\infty}H_{1}(s)ds\right)^{q-p}G_{1}^{\frac{p(1+\alpha p')}{p'}}(r)dr\\&
=\frac{1}{(1+\alpha p')^{\frac{p}{p'}}}\int_{0}^{\infty}U^{\frac{p}{q}}(r)\left(\int_{r}^{\infty}H_{3}(s)ds\right)G_{1}^{\frac{p(1+\alpha p')}{p'}}(r)\left(\int_{r}^{\infty}H_{1}(s)ds\right)^{q-p}U^{\frac{q-p}{q}}dr\\&
\geq\frac{1}{(1+\alpha p')^{\frac{p}{p'}}}\left(\int_{0}^{\infty}\left(\int_{r}^{\infty}H_{3}(s)ds\right)^{\frac{q}{p}}U(r)G_{1}^{\frac{q(1+\alpha p')}{p'}}(r)dr\right)^{\frac{p}{q}}\\&
\times\left(\int_{0}^{\infty}\left(\int_{r}^{\infty}H_{1}(s)ds\right)^{q}U(r)dr\right)^{\frac{q-p}{q}}\\&
=\frac{1}{(1+\alpha p')^{\frac{p}{p'}}}\left(\int_{0}^{\infty}U(r)\left(\int_{r}^{\infty}H_{3}(s)ds\right)^{\frac{q}{p}}G_{1}^{\frac{q(1+\alpha p')}{p'}}(r)dr\right)^{\frac{p}{q}}\\&
\times\left(\int_{\mathbb X}\left(\int_{\X\setminus B(a,|x|_{a})}g(y)z(y)dy\right)^{q}u(x)dx\right)^{\frac{q-p}{q}}\\&
=\frac{B^{\frac{q-p}{q}}}{(1+\alpha p')^{\frac{p}{p'}}}\left(\int_{0}^{\infty}U(r)\left(\int_{r}^{\infty}H_{3}(s)ds\right)^{\frac{q}{p}}G_{1}^{\frac{q(1+\alpha p')}{p'}}(r)dr\right)^{\frac{p}{q}}.
\end{align*}
Therefore,
\begin{equation}
B^{\frac{p}{q}}\geq\frac{1}{(1+\alpha p')^{\frac{p}{p'}}}\left(\int_{0}^{\infty}U(r)\left(\int_{r}^{\infty}H_{3}(s)ds\right)^{\frac{q}{p}}G_{1}^{\frac{q(1+\alpha p')}{p'}}(r)dr\right)^{\frac{p}{q}}.
\end{equation}
By using reverse Minkowski inequality with exponent $\frac{q}{p}<0$ , so that we obtain
\begin{align*}
&\left(\int_{0}^{\infty}U(r)\left(\int_{r}^{\infty}H_{3}(s)ds\right)^{\frac{q}{p}}G_{1}^{\frac{q(1+\alpha p')}{p'}}(r)dr\right)^{\frac{p}{q}}\\&
=\left(\int_{0}^{\infty}\left(\int_{r}^{\infty}U^{\frac{p}{q}}(r)H_{3}(s)G^{\frac{(1+\alpha p')p}{p'}}(r)ds\right)^{\frac{q}{p}}dr\right)^{\frac{p}{q}}\\&
=\left(\int_{0}^{\infty}\left(\int_{0}^{\infty}U^{\frac{p}{q}}(r)H_{3}(s)G_{1}^{\frac{(1+\alpha p')p}{p'}}(r)\chi_{\{r<s\}}ds\right)^{\frac{q}{p}}dr\right)^{\frac{p}{q}}\\&
\stackrel{\eqref{minkow}}\geq \int_{0}^{\infty}H_{3}(s)\left(\int_{0}^{s}U(r)G_{1}^{\frac{q(1+\alpha p')}{p'}}(r)dr\right)^{\frac{p}{q}}ds\\&
=\int_{\mathbb X}g^{p}(y)G^{-\alpha p}(y)\left(\int_{\mathbb X\setminus B(a,|y|_{a})}u(x)G^{\frac{q(1+\alpha p')}{p'}}(x)dx\right)^{\frac{p}{q}}dy\\&
\geq \widetilde{D}^{p}(\alpha)\int_{\mathbb X}g^{p}(y)dy,
\end{align*}
where $\widetilde{D}(\alpha):= \inf_{x\neq a}\widetilde{D}(x,\alpha)=\inf_{x\neq a}G^{-\alpha }(x)\left(\int_{B(a,|x|_{a})}u(y)G^{\frac{q(1+\alpha p')}{p'}}(y)dy\right)^{\frac{1}{q}}$ and $\chi$ is the cut-off function.
Then we obtain
\begin{align*}
B^{\frac{p}{q}}&=\left(\int_{\mathbb X}\left(\int_{\X \setminus B(a,|x|_{a})}f(y)dy\right)^{q}u(x)dx\right)^{\frac{p}{q}}\geq \frac{\widetilde{D}^{p}(\alpha)}{(1+\alpha p')^{\frac{p}{p'}}} \int_{\mathbb X}g^{p}(y)dy\\&
=\frac{\widetilde{D}^{p}(\alpha)}{(1+\alpha p')^{\frac{p}{p'}}} \int_{\mathbb X}f^{p}(y)v(y)dy.
\end{align*}

\textbf{Step 2.} Let us recall $D_{2}$, given in the following form:
\begin{equation}
0< D_{2}=\inf_{x\neq a}\left[\left(\int_{ B(a,|x|_{a})}u(x)dx\right)^{\frac{1}{q}}\left(\int_{\mathbb X\setminus B(a,|x|_{a})}v^{1-p'}(y)dy\right)^{\frac{1}{p'}}\right].
\end{equation}

Let us note a relation between $G$ and $G_{1}$,
\begin{align}
G(x)=\int_{\X \setminus B(a,|x|_{a})}v^{-\frac{p'}{p}}dx&=\int_{\X \setminus B(a,|x|_{a})}z^{p'}dx\nonumber\\&
=\int_{|x|_{a}}^{\infty}\int_{\sum_{r}}z^{p'}(r,\omega)\lambda(r,\omega)drd\omega\label{VV1*}\\&
=\int_{|x|_{a}}^{\infty}\tilde{H}_{2}(r)dr\nonumber\\&
=:G_{1}(|x|_{a})\nonumber.
\end{align}
For $|x|_{a}\leq |y|_{a}$, we have
$$G_{1}(|x|_{a})=\int_{|x|_{a}}^{\infty}\tilde{H}_{2}(r)dr\geq\int_{|y|_{a}}^{\infty}\tilde{H}_{2}(r)dr=G_{1}(|y|_{a}),$$
it means $G(x)\geq G(y)$.
From $\frac{q(1+\alpha p')}{p'}>0,$ we obtain
$$\int_{B(a,|x|_{a})}u(y)G^{\frac{q(1+\alpha p')}{p'}}(x)dy\geq \int_{B(a,|x|_{a})}u(y)G^{\frac{q(1+\alpha p')}{p'}}(y)dy,$$
and by using $q<0$, we have
\begin{equation}
\begin{split}
\widetilde{D}(x,\alpha)&=G^{-\alpha }(x)\left(\int_{B(a,|x|_{a})}u(y)G^{\frac{q(1+\alpha p')}{p'}}(y)dy\right)^{\frac{1}{q}}\\&
\geq G^{-\alpha }(x)\left(\int_{B(a,|x|_{a})}u(y)G^{\frac{q(1+\alpha p')}{p'}}(x)dy\right)^{\frac{1}{q}}\\&
=G^{-\alpha }(x)G^{\frac{(1+\alpha p')}{p'}}(x)\left(\int_{B(a,|x|_{a})}u(y)dy\right)^{\frac{1}{q}}\\&
=G^{\frac{1}{p'}}(x)\left(\int_{B(a,|x|_{a})}u(y)dy\right)^{\frac{1}{q}}\\&
\stackrel{1>(-\alpha p')^{-\frac{1}{q}}}\geq(-\alpha p')^{-\frac{1}{q}}G^{\frac{1}{p'}}(x)\left(\int_{B(a,|x|_{a})}u(y)dy\right)^{\frac{1}{q}}.
\end{split}
\end{equation}
Consequently,
$$\widetilde{D}(x,\alpha)\geq (-\alpha p')^{-\frac{1}{q}}D_{2},$$
it means $$\widetilde{D}(\alpha)\geq (-\alpha p')^{-\frac{1}{q}}D_{2}.$$
Finally, we obtain
$$B^{\frac{1}{q}}\geq\frac{D_{1}(-\alpha p')^{-\frac{1}{q}}}{(1+\alpha p')^{\frac{1}{p'}}} \left(\int_{\mathbb X}f^{p}(y)v(y)dy\right)^{\frac{1}{p}}.$$
Then, as in the previous case we have
$$\sup_{\alpha\in \left(0,-\frac{1}{p'}\right)}\frac{(-\alpha p')^{-\frac{1}{q}}}{(1+\alpha p')^{\frac{1}{p'}}}=\left(\frac{p'}{p'+q}\right)^{-\frac{1}{q}}\left(\frac{q}{p'+q}\right)^{-\frac{1}{p'}}.$$
Therefore, we have that the biggest constant satisfies
$$C(p,q)\geq \left(\frac{p'}{p'+q}\right)^{-\frac{1}{q}}\left(\frac{q}{p'+q}\right)^{-\frac{1}{p'}}D_{2}.$$

\textbf{Step 3.} Let us give a necessity condition of  inequality \eqref{har*}. By using \eqref{har*} and $f(x)=v^{-\frac{p'}{p}}(x)\chi_{(t,\infty)}(|x|_{a})$, where $\chi$ is cut-off function, we compute
\begin{equation}
\begin{split}
C(p,q)&\leq \left[\int_{\mathbb X}\left(\int_{\X\setminus B(a,|x|_{a})}f(y)dy\right)^{q}u(x)dx\right]^{\frac{1}{q}} \left(\int_{\mathbb X}f^{p}(y)v(y)dx\right)^{-\frac{1}{p}}\\&
= \left[\int_{\mathbb X}\left(\int_{\X\setminus B(a,|x|_{a})}f(y)dy\right)^{q}u(x)dx\right]^{\frac{1}{q}} \left(\int_{|x|_{a}\geq t}v^{-p'}(y)v(y)dx\right)^{-\frac{1}{p}}\\&
\stackrel{q<0}\leq\left[\int_{|x|_{a}\leq t}\left(\int_{|x|_{a}\geq t} v^{-\frac{p'}{p}}(y)dy\right)^{q}u(x)dx\right]^{\frac{1}{q}}\left(\int_{|x|_{a}\geq t}v^{-p'}(y)v(y)dx\right)^{-\frac{1}{p}}\\&
=\left[\int_{|x|_{a}\geq t}v^{1-p'}(y)dx\right]^{\frac{1}{p'}}\left[\int_{|x|_{a}\leq t}u(y)dy\right]^{\frac{1}{q}},
\end{split}
\end{equation}
 which gives $D_{2}\geq C(p,q)$.
\end{proof}

\section{Consequences}
In this section, we consider some consequences of the reverse integral Hardy inequality.
\subsection{Homogeneous Lie groups}
 Let us recall that a Lie group (on $\mathbb{R}^{n}$) $\mathbb{G}$ with the dilation
$$D_{\lambda}(x):=(\lambda^{\nu_{1}}x_{1},\ldots,\lambda^{\nu_{n}}x_{n}),\; \nu_{1},\ldots, \nu_{n}>0,\; D_{\lambda}:\mathbb{R}^{n}\rightarrow\mathbb{R}^{n},$$
which is an automorphism of the group $\mathbb{G}$ for each $\lambda>0,$
is called a {\em homogeneous (Lie) group}. For simplicity, throughout this paper we use the notation $\lambda x$ for the dilation $D_{\lambda}.$  The homogeneous dimension of the homogeneous group $\mathbb{G}$ is denoted by $Q:=\nu_{1}+\ldots+\nu_{n}.$
Also, in this note we denote a homogeneous quasi-norm on $\mathbb{G}$ by $|x|$, which
is a continuous non-negative function
\begin{equation}
\mathbb{G}\ni x\mapsto |x|\in[0,\infty),
\end{equation}
with the properties

\begin{itemize}
	\item[i)] $|x|=|x^{-1}|$ for all $x\in\mathbb{G}$,
	\item[ii)] $|\lambda x|=\lambda |x|$ for all $x\in \mathbb{G}$ and $\lambda>0$,
	\item[iii)] $|x|=0$ iff $x=0$.
\end{itemize}
Moreover, the following polarisation formula on homogeneous Lie groups will be used in our proofs:
there is a (unique)
positive Borel measure $\sigma$ on the
unit quasi-sphere
$
\S:=\{x\in \mathbb{G}:\,|x|=1\},
$
so that for every $f\in L^{1}(\mathbb{G})$ we have
\begin{equation}\label{EQ:polar}
\int_{\mathbb{G}}f(x)dx=\int_{0}^{\infty}
\int_{\S}f(ry)r^{Q-1}d\sigma(y)dr.
\end{equation}
We refer to \cite{FS1} for the original appearance of such groups, and to \cite{FR} for a recent comprehensive treatment.
Let us define the quasi-ball centered at $x$ with radius $r$ in the following form:
\begin{equation}
B(x,r):=\{x\in\mathbb{G}:|x^{-1}y|<r\}.
\end{equation}
Then we have the following reverse integral Hardy inequality on homogeneous Lie groups.
\begin{cor}
Let $\mathbb{G}$ be a homogeneous Lie group of homogeneous dimension $Q$ with a quasi-norm $|\cdot|$. Assume that $q<0$, $p\in(0,1)$ and $\alpha,\beta\in\mathbb{R}$. Then the reverse integral Hardy inequality
\begin{equation}\label{harhom}
\left[\int_{\mathbb{G}}\left(\int_{B(0,|x|)}f(y)dy\right)^{q}|x|^{\alpha}dx\right]^{\frac{1}{q}}\geq C\left(\int_{\mathbb{G}}f^{p}(x)|x|^{\beta}dx\right)^{\frac{1}{p}},
\end{equation}
holds for $C>0$ and for all non-negative measurable functions $f$, if only if
\begin{equation}
\alpha+Q<0,\,\,\, \beta(1-p')+Q>0\,\,\, \text{and}\,\,\, \frac{Q+\alpha}{q}+\frac{Q+\beta(1-p')}{p'}=0.
\end{equation}
Moreover, the biggest constant $C$ for \eqref{harhom} satisfies
\begin{equation}
\begin{split}
&\left(\frac{|\S|}{|\alpha+Q|}\right)^{\frac{1}{q}}\left(\frac{|\S|}{Q+\beta(1-p')}\right)^{\frac{1}{p'}}\geq C\\&
\geq\left(\frac{|\S|}{|\alpha+Q|}\right)^{\frac{1}{q}}\left(\frac{|\S|}{Q+\beta(1-p')}\right)^{\frac{1}{p'}}\left(\frac{p'}{p'+q}\right)^{-\frac{1}{q}}\left(\frac{q}{p'+q}\right)^{-\frac{1}{p'}},
\end{split}
\end{equation}
where $|\mathfrak{S}|$ is the area of unit sphere with respect to $|\cdot|$.
\end{cor}
\begin{proof}
Let us check condition \eqref{D1} with $u(x)=|x|^{\alpha}$, $v(x)=|x|^{\beta}$ and with $a=0$. Let us calculate the first integral in \eqref{D1}:
\begin{equation}
\begin{split}
\int_{\mathbb{G}\setminus B(0,|x|)}u(y)dy&=\int_{\mathbb{G}\setminus B(0,|x|)}|y|^{\alpha}dy
\stackrel{\eqref{EQ:polar}}=\int_{|x|}^{\infty}\int_{\S}\rho^{\alpha}\rho^{Q-1}d\rho d\sigma(\omega)\\&
=|\S|\int_{|x|}^{\infty}\rho^{Q+\alpha-1}d\rho
\stackrel{Q+\alpha<0}=-\frac{|\S|}{Q+\alpha}|x|^{Q+\alpha}=\frac{|\S|}{|Q+\alpha|}|x|^{Q+\alpha},
\end{split}
\end{equation}
where $|\S|$ is the area of the unit quasi-sphere in $\mathbb{G}$.
Then,
\begin{equation}
\begin{split}
\int_{B(0,|x|)}v^{1-p'}(y)dy&=\int_{B(0,|x|)}|y|^{\beta(1-p')}dy
\stackrel{\eqref{EQ:polar}}=\int_{0}^{|x|}\int_{\S}\rho^{\beta(1-p')}\rho^{Q-1}d\rho d\sigma(\omega)\\&
=|\S|\int_{0}^{|x|}\rho^{Q+\beta(1-p')-1}d\rho\\&
\stackrel{Q+\beta(1-p')>0}=\frac{|\S|}{Q+\beta(1-p')}|x|^{Q+\beta(1-p')}.
\end{split}
\end{equation}
Finally by summarising above facts with $\frac{Q+\alpha}{q}+\frac{Q+\beta(1-p')}{p'}=0$, we have
\begin{equation}
\begin{split}
D_{1}&=\left(\frac{|\S|}{|\alpha+Q|}\right)^{\frac{1}{q}}\left(\frac{|\S|}{Q+\beta(1-p')}\right)^{\frac{1}{p'}}\inf_{r>0}r^{\frac{Q+\alpha}{q}+\frac{Q+\beta(1-p')}{p'}}\\&
=\left(\frac{|\S|}{|\alpha+Q|}\right)^{\frac{1}{q}}\left(\frac{|\S|}{Q+\beta(1-p')}\right)^{\frac{1}{p'}}>0.
\end{split}
\end{equation}
Then by using \eqref{C}, we obtain
\begin{equation}
\begin{split}
&\left(\frac{|\S|}{|\alpha+Q|}\right)^{\frac{1}{q}}\left(\frac{|\S|}{Q+\beta(1-p')}\right)^{\frac{1}{p'}}\geq C\\&
\geq\left(\frac{|\S|}{|\alpha+Q|}\right)^{\frac{1}{q}}\left(\frac{|\S|}{Q+\beta(1-p')}\right)^{\frac{1}{p'}}\left(\frac{p'}{p'+q}\right)^{-\frac{1}{q}}\left(\frac{q}{p'+q}\right)^{-\frac{1}{p'}},
\end{split}
\end{equation}
completing the proof.
\end{proof}
Then we have conjugate reverse integral Hardy inequality on homogeneous Lie groups.
\begin{cor}
Let $\mathbb{G}$ be a homogeneous Lie group of homogeneous dimension $Q$ with a quasi-norm $|\cdot|$. Assume that $q<0$, $p\in(0,1)$ and $\alpha,\beta\in\mathbb{R}$. Then the conjugate reverse integral Hardy inequality
\begin{equation}\label{harhom*}
\left[\int_{\mathbb{G}}\left(\int_{\mathbb{G}\setminus B(0,|x|)}f(y)dy\right)^{q}|x|^{\alpha}dx\right]^{\frac{1}{q}}\geq C \left(\int_{\mathbb{G}}f^{p}(x)|x|^{\beta}dx\right)^{\frac{1}{p}},
\end{equation}
holds for $C>0$ and for all non-negative measurable functions $f$, if only if
\begin{equation}
\alpha+Q>0,\,\,\, \beta(1-p')+Q<0\,\,\, \text{and}\,\,\, \frac{Q+\alpha}{q}+\frac{Q+\beta(1-p')}{p'}=0.
\end{equation}
 Moreover, the biggest constant $C$ for \eqref{harhom*} satisfies
\begin{equation}
\begin{split}
&\left(\frac{|\S|}{\alpha+Q}\right)^{\frac{1}{q}}\left(\frac{|\S|}{|Q+\beta(1-p')|}\right)^{\frac{1}{p'}}\geq C\\&
\geq\left(\frac{|\S|}{\alpha+Q}\right)^{\frac{1}{q}}\left(\frac{|\S|}{|Q+\beta(1-p')|}\right)^{\frac{1}{p'}}\left(\frac{p'}{p'+q}\right)^{-\frac{1}{q}}\left(\frac{q}{p'+q}\right)^{-\frac{1}{p'}},
\end{split}
\end{equation}
where $|\mathfrak{S}|$ is the area of unit sphere with respect to $|\cdot|$.
\end{cor}
\begin{proof}
Proof of this corollary is similar to the previous case.
\end{proof}
\subsection{Hyperbolic space}
Let $\mathbb{H}^{n}$ be the hyperbolic space of dimension $n$ and let $a\in\mathbb{H}^{n}$. Let us set
\begin{equation}
u(x)=(\sinh |x|_{a})^{\alpha},\,\,\,\,\,v(x)=(\sinh |x|_{a})^{\beta}.
\end{equation}
Then we have the following result of this subsection.
\begin{cor}
Let $\H$ be the hyperbolic space of dimension $n$ and let $a\in\H$. Assume that $q<0$, $p\in(0,1)$ and $\alpha,\beta\in\mathbb{R}$. Then the reverse integral Hardy inequality
\begin{equation}
\left[\int_{\H}\left(\int_{B(a,|x|_{a})}f(y)dy\right)^{q}(\sinh |x|_{a})^{\alpha}dx\right]^{\frac{1}{q}}\geq C\left(\int_{\H}f^{p}(x)(\sinh |x|_{a})^{\beta}dx\right)^{\frac{1}{p}},
\end{equation}
holds for $C>0$ and for all non-negative measurable functions $f,$ if
\begin{equation}
0\leq\alpha+n<1,\,\,\, \beta(1-p')+n>0\,\,\, \text{and}\,\,\, \frac{\alpha+n}{q}+\frac{\beta(1-p')+n}{p'}\geq\frac{1}{q}+\frac{1}{p'}.
\end{equation}
\end{cor}
\begin{proof}
Let us check condition \eqref{D1}. By using polar decomposition for the hyperbolic space, we have
\begin{equation}\label{integ}
D_{1}=\inf_{x\neq a}\left(\int_{|x|_{a}}^{\infty}(\sinh \rho)^{\alpha+n-1}d\rho\right)^{\frac{1}{q}}\left(\int_{0}^{|x|_{a}}(\sinh \rho)^{\beta(1-p')+n-1}d\rho\right)^{\frac{1}{p'}}.
\end{equation}
If $\alpha+n<1$ and  $\beta(1-p')+n>0$, then \eqref{integ} is integrable. Let us check the finiteness and positiveness of the infimum \eqref{integ}. Let us divide the proof in two cases.

First case, $|x|_{a}\gg1$. Then $\sinh |x|_{a}\approx \exp |x|_{a}$ if $|x|_{a}\gg1$. Then we obtain,
\begin{equation}
\begin{split}
D^{1}_{1}&=\inf_{|x|_{a}\gg1}\left(\int_{|x|_{a}}^{\infty}(\sinh \rho)^{\alpha+n-1}d\rho\right)^{\frac{1}{q}}\left(\int_{0}^{|x|_{a}}(\sinh \rho)^{\beta(1-p')+n-1}d\rho\right)^{\frac{1}{p'}}\\&
\simeq \inf_{|x|_{a}\gg1}\left(\int_{|x|_{a}}^{\infty}(\exp \rho)^{\alpha+n-1}d\rho\right)^{\frac{1}{q}}\left(\int_{0}^{|x|_{a}}(\exp \rho)^{\beta(1-p')+n-1}d\rho\right)^{\frac{1}{p'}}\\&
=\inf_{|x|_{a}\gg1}\left((\exp|x|_{a})^{\alpha+n-1}\right)^{\frac{1}{q}}\left((\exp|x|_{a})^{\beta(1-p')+n-1}\right)^{\frac{1}{p'}}\\&
=\inf_{|x|_{a}\gg1}(\exp|x|_{a})^{\frac{\alpha+n-1}{q}+\frac{\beta(1-p')+n-1}{p'}},
\end{split}
\end{equation}
infimum of the last term is positive, if only if $\frac{\alpha+n-1}{q}+\frac{\beta(1-p')+n-1}{p'}\geq0$, i.e., $\frac{\alpha+n}{q}+\frac{\beta(1-p')+n}{p'}\geq\frac{1}{q}+\frac{1}{p'}$,  then $D^{1}_{1}>0$.

Let us consider the another case $|x|_{a}\ll1$. For $|x|_{a}\ll1$ we have $\sinh \rho_{\{0\leq\rho<|x|_{a} \}}\approx\rho,$ then we calculate
\begin{equation}
\begin{split}
\inf_{|x|_{a}\ll1}&\left(\int_{|x|_{a}}^{\infty}(\sinh \rho)^{\alpha+n-1}d\rho\right)^{\frac{1}{q}}\left(\int^{|x|_{a}}_{0}(\sinh \rho)^{\beta(1-p')+n-1}d\rho\right)^{\frac{1}{p'}}\\&
\simeq\inf_{|x|_{a}\ll1}\left(\int_{|x|_{a}}^{R}(\sinh \rho)^{\alpha+n-1}d\rho+\int_{R}^{\infty}(\sinh \rho)^{\alpha+n-1}d\rho\right)^{\frac{1}{q}}\left(\int^{|x|_{a}}_{0}\rho^{\beta(1-p')+n-1}d\rho\right)^{\frac{1}{p'}}\\&
\simeq\inf_{|x|_{a}\ll1}\left(\int_{|x|_{a}}^{R}(\sinh \rho)^{\alpha+n-1}d\rho+\int_{R}^{\infty}(\sinh \rho)^{\alpha+n-1}d\rho\right)^{\frac{1}{q}}|x|_{a}^{\frac{\beta(1-p')+n}{p'}}.
\end{split}
\end{equation}
Similarly, for small $R$ we have $\sinh \rho_{\{|x|_{a}\leq\rho<R \}}\approx\rho,$ so that we obtain
\begin{equation}
\begin{split}
\inf_{|x|_{a}\ll1}&\left(\int_{|x|_{a}}^{\infty}(\sinh \rho)^{\alpha+n-1}d\rho\right)^{\frac{1}{q}}\left(\int^{|x|_{a}}_{0}(\sinh \rho)^{\beta(1-p')+n-1}d\rho\right)^{\frac{1}{p'}}\\&
\simeq\inf_{|x|_{a}\ll1}\left(\int_{|x|_{a}}^{R}(\sinh \rho)^{\alpha+n-1}d\rho+\int_{R}^{\infty}(\sinh \rho)^{\alpha+n-1}d\rho\right)^{\frac{1}{q}}|x|_{a}^{\frac{\beta(1-p')+n}{p'}}\\&
\simeq \inf_{|x|_{a}\ll1}\left(\int_{|x|_{a}}^{R} \rho^{\alpha+n-1}d\rho+\int_{R}^{\infty}(\sinh \rho)^{\alpha+n-1}d\rho\right)^{\frac{1}{q}}|x|_{a}^{\frac{\beta(1-p')+n}{p'}}\\&
\simeq \inf_{|x|_{a}\ll1}\left( |x|_{a}^{\alpha+n}+C_{R}\right)^{\frac{1}{q}}|x|_{a}^{\frac{\beta(1-p')+n}{p'}}.
\end{split}
\end{equation}
If $\alpha+n\geq0$, we have $\frac{\alpha+n}{q}\leq0$, then we have
\begin{equation}
\begin{split}
D^{2}_{1}&=\inf_{|x|_{a}\ll1}\left(\int_{|x|_{a}}^{\infty}(\sinh \rho)^{\alpha+n-1}d\rho\right)^{\frac{1}{q}}\left(\int^{|x|_{a}}_{0}(\sinh \rho)^{\beta(1-p')+n-1}d\rho\right)^{\frac{1}{p'}}\\&
\simeq \inf_{|x|_{a}\ll1}\left( |x|_{a}^{\alpha+n}+C_{R}\right)^{\frac{1}{q}}|x|_{a}^{\frac{\beta(1-p')+n}{p'}}\\&
\simeq \inf_{|x|_{a}\ll1}|x|_{a}^{\frac{\beta(1-p')+n}{p'}}>0,
\end{split}
\end{equation}
and infimum is positive, if only if $\frac{\beta(1-p')+n}{p'}<0$, i.e., $\beta(1-p')+n>0$.
\end{proof}
Let us give the reverse conjugate integral Hardy's inequality in hyperbolic spaces:
\begin{cor}
Let $\H$ be the hyperbolic space of  dimension $n$ and $a\in \H$. Assume that $q<0$, $p\in(0,1)$ and let $\alpha,\beta\in\mathbb{R}$. Then the reverse conjugate integral Hardy inequality
\begin{equation}
\left[\int_{\H}\left(\int_{\mathbb X\setminus B(a,|x|_{a})}f(y)dy\right)^{q}(\sinh |x|_{a})^{\alpha}dx\right]^{\frac{1}{q}}\geq C\left(\int_{\H}f^{p}(x)(\sinh |x|_{a})^{\beta}dx\right)^{\frac{1}{p}},
\end{equation}
holds for all non-negative measurable functions $f,$ if
$$\alpha+n>0,\,\,\,1>\beta(1-p')+n\geq0\,\,\, \text{and}\,\,\, \frac{\alpha+n}{q}+\frac{\beta(1-p')+n}{p'}\geq\frac{1}{q}+\frac{1}{p'}.$$
\end{cor}
\begin{proof}
Similarly to the previous case, check condition \eqref{D2} and then, we have
\begin{equation}\label{integ*}
D_{2}=\inf_{x\neq a}\left(\int_{0}^{|x|_{a}}(\sinh \rho)^{\alpha+n-1}d\rho\right)^{\frac{1}{q}}\left(\int_{|x|_{a}}^{\infty}(\sinh \rho)^{\beta(1-p')+n-1}d\rho\right)^{\frac{1}{p'}}.
\end{equation}
If $\alpha+n>0$ and  $\beta(1-p')+n<1$, then \eqref{integ*} is integrable.
If $|x|_{a}\gg1$, we obtain,
\begin{equation}
\begin{split}
D^{1}_{2}&=\inf_{|x|_{a}\gg1}\left(\int_{|x|_{a}}^{\infty}(\sinh \rho)^{\alpha+n-1}d\rho\right)^{\frac{1}{q}}\left(\int_{0}^{|x|_{a}}(\sinh \rho)^{\beta(1-p')+n-1}d\rho\right)^{\frac{1}{p'}}\\&
\simeq \inf_{|x|_{a}\gg1}\left(\int_{|x|_{a}}^{\infty}(\exp \rho)^{\alpha+n-1}d\rho\right)^{\frac{1}{q}}\left(\int_{0}^{|x|_{a}}(\exp \rho)^{\beta(1-p')+n-1}d\rho\right)^{\frac{1}{p'}}\\&
=\inf_{|x|_{a}\gg1}\left((\exp|x|_{a})^{\alpha+n-1}\right)^{\frac{1}{q}}\left((\exp|x|_{a})^{\beta(1-p')+n-1}\right)^{\frac{1}{p'}}\\&
=\inf_{|x|_{a}\gg1}(\exp|x|_{a})^{\frac{\alpha+n-1}{q}+\frac{\beta(1-p')+n-1}{p'}},
\end{split}
\end{equation}
infimum of the last term is positive, if only if $\frac{\alpha+n-1}{q}+\frac{\beta(1-p')+n-1}{p'}\geq0$, i.e., $\frac{\alpha+n}{q}+\frac{\beta(1-p')+n}{p'}\geq\frac{1}{q}+\frac{1}{p'}$,  then $D^{1}_{1}>0$.

If $|x|_{a}\ll1$, we obtain
\begin{equation}
\begin{split}
\inf_{|x|_{a}\ll1}&\left(\int_{0}^{|x|_{a}}(\sinh \rho)^{\alpha+n-1}d\rho\right)^{\frac{1}{q}}\left(\int^{\infty}_{|x|_{a}}(\sinh \rho)^{\beta(1-p')+n-1}d\rho\right)^{\frac{1}{p'}}\\&
\simeq\inf_{|x|_{a}\ll1}\left(\int_{0}^{|x|_{a}}\rho^{\alpha+n-1}d\rho\right)^{\frac{1}{q}} \left(\int_{|x|_{a}}^{R}(\sinh \rho)^{\beta(1-p')+n-1}d\rho +\int_{R}^{\infty}(\sinh \rho)^{\beta(1-p')+n-1}d\rho\right)^{\frac{1}{p'}}\\&
\simeq\inf_{|x|_{a}\ll1}\left(\int_{|x|_{a}}^{R}(\sinh \rho)^{\beta(1-p')+n-1}d\rho +\int_{R}^{\infty}(\sinh \rho)^{\beta(1-p')+n-1}d\rho\right)^{\frac{1}{p'}}|x|_{a}^{\frac{\alpha+n}{q}}.
\end{split}
\end{equation}
Similarly, for small $R$ we have $\sinh \rho_{\{|x|_{a}\leq\rho<R \}}\approx\rho,$ so that we obtain
\begin{equation}
\begin{split}
\inf_{|x|_{a}\ll1}&\left(\int_{0}^{|x|_{a}}(\sinh \rho)^{\alpha+n-1}d\rho\right)^{\frac{1}{q}}\left(\int^{\infty}_{|x|_{a}}(\sinh \rho)^{\beta(1-p')+n-1}d\rho\right)^{\frac{1}{p'}}\\&
\simeq \inf_{|x|_{a}\ll1}\left(\int_{|x|_{a}}^{R}(\sinh \rho)^{\beta(1-p')+n-1}d\rho +\int_{R}^{\infty}(\sinh \rho)^{\beta(1-p')+n-1}d\rho\right)^{\frac{1}{p'}}|x|_{a}^{\frac{\alpha+n}{q}}\\&
\simeq \inf_{|x|_{a}\ll1}\left( |x|_{a}^{\beta(1-p')+n}+C'_{R}\right)^{\frac{1}{q}}|x|_{a}^{\frac{\alpha+n}{q}}.
\end{split}
\end{equation}
If $\beta(1-p')+n\geq0$, we have $\frac{\beta(1-p')+n}{q}\leq0$, then we have
\begin{equation}
\begin{split}
D^{2}_{2}&=\inf_{|x|_{a}\ll1}\left(\int_{0}^{|x|_{a}}(\sinh \rho)^{\alpha+n-1}d\rho\right)^{\frac{1}{q}}\left(\int^{\infty}_{|x|_{a}}(\sinh \rho)^{\beta(1-p')+n-1}d\rho\right)^{\frac{1}{p'}}\\&
\simeq \inf_{|x|_{a}\ll1}\left( |x|_{a}^{\beta(1-p')+n}+C'_{R}\right)^{\frac{1}{q}}|x|_{a}^{\frac{\alpha+n}{q}}\\&
\simeq \inf_{|x|_{a}\ll1}|x|_{a}^{\frac{\alpha+n}{q}},
\end{split}
\end{equation}
and infimum is positive, if only if $\frac{\alpha+n}{q}<0$, i.e., $\alpha+n>0$.
\end{proof}
\subsection{Cartan-Hadamard manifolds}
Let $(M,g)$ be the Cartan-Hadamard manifold with curvature $K_{M}$. If $K_{M}=0$  then $J(t,\omega)=1$ and we set
\begin{equation}
u(x)=|x|_{a}^{\alpha},\,\,\,\,\,v(x)=|x|_{a}^{\beta},\,\,\, \text{when}\,\,\,K_{M}=0.
\end{equation}
If $ K_{M}<0$ then $J(t,\omega)=\left(\frac{\sinh \sqrt{b}t}{\sqrt{b}t}\right)^{n-1}$ and we set
\begin{equation}
u(x)=(\sinh\sqrt{-K_{M}} |x|_{a})^{\alpha},\,\,\,\,\,v(x)=(\sinh\sqrt{-K_{M}}|x|_{a})^{\beta},\,\,\, \text{when}\,\,\,K_{M}<0.
\end{equation}
Then we have the following result of this subsection.
\begin{cor}
Assume that $(M,g)$ be the Cartan-Hadamard manifold of dimension $n$ and with curvature $K_{M}$. Assume that $q<0$, $p\in(0,1)$ and $\alpha,\beta\in\mathbb{R}$. Then we have
\begin{itemize}
\item[i)] if $K_{M}=0$, $u(x)=|x|_{a}^{\alpha},v(x)=|x|_{a}^{\beta}$, then
\begin{equation}
\left[\int_{M}\left(\int_{B(a,|x|_{a})}f(y)dy\right)^{q}|x|_{a}^{\alpha}dx\right]^{\frac{1}{q}}\geq C\left(\int_{M}f^{p}(x)|x|^{\beta}dx\right)^{\frac{1}{p}},
\end{equation}
holds for $C>0$ and for non-negative measurable functions $f,$ if only if $\alpha+n<0$, $\beta(1-p')+n>0$ and $\frac{n+\alpha}{q}+\frac{n+\beta(1-p')}{p'}=0$;
\item[ii)]if $K_{M}=0$, $u(x)=|x|_{a}^{\alpha},v(x)=|x|_{a}^{\beta}$, then
\begin{equation}
\left[\int_{M}\left(\int_{M\setminus B(a,|x|_{a})}f(y)dy\right)^{q}|x|^{\alpha}dx\right]^{\frac{1}{q}}\geq C\left(\int_{M}f^{p}(x)|x|^{\beta}dx\right)^{\frac{1}{p}},
\end{equation}
holds for $C>0$ and for non-negative measurable functions $f,$ if only if  $\alpha+n>0$, $\beta(1-p')+n<0$ and $\frac{n+\alpha}{q}+\frac{n+\beta(1-p')}{p'}=0$;

\item[iii)]if $K_{M}<0$, $u(x)=(\sinh\sqrt{-K_{M}} |x|_{a})^{\alpha},v(x)=(\sinh|x|_{a})^{\beta}$, then
\begin{multline}
\left[\int_{M}\left(\int_{B(a,|x|_{a})}f(y)dy\right)^{q}(\sinh\sqrt{-K_{M}} |x|_{a})^{\alpha}dx\right]^{\frac{1}{q}}\\
\geq C\left(\int_{M}f^{p}(x)(\sinh\sqrt{-K_{M}} |x|_{a})^{\beta}dx\right)^{\frac{1}{p}},
\end{multline}
holds for $C>0$ and for all non-negative measurable functions $f,$ if   $0\leq\alpha+n<1$, $\beta(1-p')+n>0$ and $\frac{\alpha+n}{q}+\frac{\beta(1-p')+n}{p'}\geq\frac{1}{q}+\frac{1}{p'}$;
\item[iv)]if $K_{M}<0$, $u(x)=(\sinh\sqrt{-K_{M}} |x|_{a})^{\alpha},v(x)=(\sinh\sqrt{-K_{M}} |x|_{a})^{\beta}$, then
\begin{multline}
\left[\int_{M}\left(\int_{M\setminus B(a,|x|_{a})}f(y)dy\right)^{q}(\sinh\sqrt{-K_{M}} |x|_{a})^{\alpha}dx\right]^{\frac{1}{q}}\\
\geq C\left(\int_{M}f^{p}(x)(\sinh\sqrt{-K_{M}} |x|_{a})^{\beta}dx\right)^{\frac{1}{p}},
\end{multline}
holds for $C>0$ and for all non-negative measurable functions $f,$ if  $\alpha+n>0$,  $1>\beta(1-p')+n\geq0$ and $\frac{\alpha+n}{q}+\frac{\beta(1-p')+n}{p'}\geq\frac{1}{q}+\frac{1}{p'}$.

\end{itemize}

\end{cor}

\end{document}